\documentclass[english,11pt, reqno]{amsart}
\usepackage{amsmath}
\usepackage{amssymb}
\usepackage{amsthm}
\usepackage{bbm}
\usepackage{pgfplots}
\usepackage[hidelinks]{hyperref}
\usepackage[nameinlink]{cleveref}
\usepackage[margin=1in]{geometry}
\usepackage{multicol}
\usepackage{etoolbox}
\usepackage{float}
\usepackage[foot]{amsaddr}

\usetikzlibrary{intersections}
\usetikzlibrary{calc}
\usetikzlibrary{decorations.markings}

\allowdisplaybreaks

\usepackage[sc]{mathpazo}

\newcommand{\hair}{\ifmmode\mskip1mu\else\kern0.08em\fi}
\renewcommand{\P}{\mathbb{P}}

\newcommand{\R}{\mathbb{R}}

\newcommand{\N}{\mathbb{N}}
\newcommand{\Z}{\mathbb{Z}}

\newcommand{\one}{\mathbbm{1}}

\newcommand{\U}{\mathcal{U}}

\newcommand{\cytrm}[1]{\mathrm{CyTRM}(#1)}

\newtheorem{theorem}{Theorem}[section]
\newtheorem*{theorem*}{Theorem}
\newtheorem*{prop*}{Proposition}
\newtheorem{prop}[theorem]{Proposition}
\newtheorem*{corollary*}{Corollary}

\newtheorem{lemma}[theorem]{Lemma}

\theoremstyle{definition}
\newtheorem{defn}[theorem]{Definition}

\newtheorem*{notation}{Notation}
\newtheorem{remark}[theorem]{Remark}

\makeatletter
\renewenvironment{proof}[1][\proofname] {\par\pushQED{\qed}\normalfont\topsep6\p@\@plus6\p@\relax\trivlist\item[\hskip\labelsep\bfseries#1\@addpunct{.}]\ignorespaces}{\popQED\endtrivlist\@endpefalse}
\makeatother

\title{Critical point for infinite cycles in a random loop model on trees}
\author{Alan Hammond and Milind Hegde}
\date{}
\address{Department of Mathematics\\
 UC Berkeley \\
  899 Evans Hall \& 741 Evans Hall}
\email{alanmh@berkeley.edu, milind.hegde@berkeley.edu}
\thanks{The first author is supported by NSF grant DMS-1512908.}

\begin{document}
\maketitle

\begin{abstract}
We study a spatial model of random permutations on trees with a time parameter $T>0$, a special case of which is the random stirring process. The model on trees was first analysed by Bj\"ornberg and Ueltschi \cite{ueltschi}, who established the existence of infinite cycles for $T$ slightly above a putatively identified critical value but left open behaviour at arbitrarily high values of $T$. We show the existence of infinite cycles for all $T$ greater than a constant, thus classifying behaviour for all values of $T$ and establishing the existence of a sharp phase transition. Numerical studies \cite{numeric_study} of the model on $\Z^d$ have shown behaviour with strong similarities to what is proven for trees.
%
\end{abstract}

\section{Introduction}


Consider a collection of points scattered 
independently in a large three-dimensional torus so that any unit-volume region contains a unit order of points. For $T > 0$ given, 
the points follow for time~$T$ Brownian trajectories in the torus, with a short-range repulsive force continually acting between any pair of points. The system is conditioned on the collective return at time~$T$ of the particles to their starting
locations.  A random permutation is obtained by following the trajectory for time~$T$ of any given particle from its initial location. This mathematical {\em spatial random permutation} model is physically significant, as first recognised by Richard Feynman in~\cite{feynman}:
at higher values of time $T$, large cycles may be expected to form in the random permutation, with the reciprocal values $T^{-1}$ corresponding to lower temperatures at which gases such as helium form special states such as Bose-Einstein condensates or superfluids.

A simple mathematical model of spatial random permutations is the random stirring process, sometimes known as the random interchange model.   
 The model was introduced by Harris~\cite{harris}. It associates to a given graph $G = (V, E)$ a stochastic process $\big( \sigma_t : t \in [0,\infty) \big)$ which takes values in the space of permutations of the vertex set $V$. 
Each edge $e \in E$ is independently equipped with a Poisson process of rate one. 
We set $\sigma_0$ to be the identity permutation. 
 If the Poisson process on an edge $e=(v,w)$ rings at time $t$, we right-compose $\sigma_t$ with the transposition $(v,w)$, i.e., we swap $v$ and $w$ in the permutation process, in a right-continuous manner. 
The model is easily seen to be well defined on regular infinite graphs, or indeed if the maximum degree of the graph is finite. The formation of large cycles in $\sigma_t$ is a main topic of inquiry.

  B\'alint T\'oth~\cite{toth} showed that  the quantum Heisenberg ferromagnet
has    long range order and spontaneous magnetisation in a phase that corresponds to the appearance of macroscopic cycles in a variant of the random stirring model in which permutations are reweighted by a factor of two for each cycle.


T\'oth conjectured in the 1990s that the random stirring process on transient graphs (such as $\Z^d$ for $d \geq 3$) exhibit a critical point 
above which infinite cycles almost surely appear and below which they do not.

When $d$ is high, the model on regular trees of degree $d$ may be expected to be similar to the model on $\Z^d$. 
Omer Angel \cite{angel} showed that on regular trees with degree at least five, there exists a certain bounded interval of times where $\sigma_t$ a.s. has an infinite cycle. The existence of a critical value for infinite cycles 
 was proved in \cite{hammond1, hammond2}.


Aizenman and Nachtergaele in \cite{aizenman} introduced a representation of the quantum Heisenberg 
antiferromagnet via a variant of the random stirring process in which a certain time reversal occurs when particles are transposed due to the ringing of Poisson clocks on the associated edge.
Ueltschi introduced in \cite{ueltschi-reverse} a hybrid model in which ferromagnetic and antiferromagnetic effects are both present, and with Bj\"ornberg in \cite{ueltschi} analysed the new model when the underlying graph is a high degree tree; comparison with numerical evidence shows that the model is a very good surrogate for its $\Z^d$ counterpart.  The present article develops their work by proving a result that was very strongly suggested by their results: that the hybrid model has a critical point for the formation of infinite cycles.




\subsection{The Cyclic Time Random Walk and its Modification}

In order to explain Bj\"ornberg and Ueltschi's results and how our paper develops them, it is useful to begin by recalling a random process which is a very close cousin of the random stirring process. This is the Cyclic-Time Random Meander (CyTRM); it is a slight variant of the cyclic-time random walk considered by Angel in~\cite{angel} and has been used in Angel and Hammond's analyses of the random stirring process on trees.

Recall that, when the random stirring process is specified, a graph $G = (V,E)$
is given. 
 The CyTRM is defined by fixing a  parameter $T \in (0,\infty)$, and 
associating to each edge of $G$ an independent Poisson process of rate one on $[0, T)$.
We may picture the graph's vertices 
as points in the plane, with a vertical pole rising to height $T$
above each of them. 
A horizontal bridge is placed between the poles rising from vertices $v$ and $w$ at any height at which the Poisson process for the edge $(v,w)$ rings; in this case we say the edge $(v,w)$ \emph{supports} a bridge; see Figure~\ref{fig.cytrm}.

\begin{figure}[h]
\centering
\begin{tikzpicture}[scale = 1.5]

\draw[dashed] (0,0) -- ++(-1.2, -1.5);
\draw[dashed] (0,0) -- ++(1.2, -1.5);

\draw (0,0) -- ++(0, 1.5);
\draw (1.2,-1.5) -- ++(0, 1.5);
\draw[decorate,decoration={brace,amplitude=10pt, mirror},xshift=2pt,yshift=0pt] (1.2,-1.5) -- ++(0, 1.5)  node [black,midway,xshift=0.6cm] {$T$};
\draw (-1.2, -1.5) -- ++(0, 1.5);

\node[anchor = east] at (0,0) {$\phi$};

\fill (0,0) circle [radius=1pt];
\fill (1.2,-1.5) circle [radius=1pt];
\fill (-1.2,-1.5) circle [radius=1pt];

\draw (0,0.4) -- ++(1.2,-1.5);
\draw (0,1) -- ++(-1.2,-1.5);

\coordinate (left child) at (-1.2,-0.8) {};
\coordinate (future child) at (-2.4,-2.3) {};

\clip (-1.5,1.5) rectangle (1.6,-1.6);
\draw (left child.center) -- ($(left child)!0.5cm!(future child)$);
\draw[dotted] (-1.2,-0.8) -- (-2.4,-2.3);
%
\end{tikzpicture}
\caption{ An illustration of the root $\phi$ and two offspring. The vertical lines are the poles, the dashed lines are the edges in the underlying graph, and the solid slanted lines are the bridges supported by the edges underneath at the heights dictated by the independent rate one Poisson processes associated with each edge.}\label{fig.cytrm}
\end{figure}
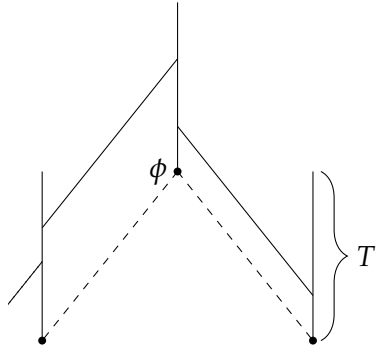
The CyTRM is a right-continuous random process $X$ mapping $[0,\infty)$
to $V \times [0,T)$ and may be depicted as a point moving in the union of the poles. If initially $X(0) = (v,0)$ for some $v \in V$, $X$ rises vertically at unit speed on the pole at $v$. If it encounters a bridge's intersection with this pole,
the process instantaneously jumps across the bridge, and then continues its unit speed ascent on the newly encountered pole. When, at time $T$, the point reaches the top of a pole, it immediately jumps to the base of the same pole. Vertical ascent then continues, so that the process' height at times $t \geq 0$ is the cyclic function $t \, \, {\rm mod} \, \, T$.
 
 The random stirring process $\sigma_T$ at parameter $T  \in  (0,\infty)$ is formed from cyclic-time random meander $X$  as the permutation on vertices induced by the  evolution of $X$ during the interval $[0,T]$. Formally, when  
 $X(0) = (v,0)$ for given $v \in V$, we have that $X(T) = \big( \sigma_T(v),0 \big)$. 
 The presence of an infinite cycle containing $v \in V$ in  $\sigma_T$ is characterised by the absence of return to its starting point by cyclic-time random meander with $X(0) = (v,0)$; see \cite{angel} for details.



The random loop model introduced by  Bj\"ornberg and Ueltschi in \cite{ueltschi} is a generalisation of cyclic-time random meander. Given a parameter $u \in [0,1]$, and the structure that specifies the meander,  independently assign to each bridge a Bernoulli random variable of parameter $u$. When the random variable equals one, the bridge is replaced by a {\em cross}; 
and by a {\em double bar} in the other case. In keeping with our previous terminology, we will refer to crosses and double bars collectively as bridges.
In this way, a collection of crosses and double bars connect poles at various heights. Associated to this system is an altered cyclic-time random meander, denoted by $X_{u,T}$, which is governed by similar rules as its precursor, 
with the behaviour of the meander when a cross is encountered being the same as when a bridge was encountered in the existing model. The difference is that, on jumping over a double bar, unit speed motion along the new pole occurs in the {\em opposite} direction to that adopted by the meander immediately before the jump; see Figure~\ref{fig.bridge-types}. As before, the model is well defined for graphs with bounded degree. When $u=1$, then, we recover the original model.

\begin{figure}[h]
\centering
\begin{tikzpicture}
[scale=0.9,
decoration={
markings,
mark=at position 0.175 with {\arrow{stealth}},
mark=at position 0.4 with {\arrow{stealth}},
mark=at position 0.6 with {\arrow{stealth}},
mark=at position 0.85 with {\arrow{stealth}}}
]

\draw[semithick] (0,0) -- ++(0, -2) -- ++ (2,-0.15) -- ++(0,-2);
\draw[semithick,postaction={decorate}] (0,-4.15) -- ++(0, 2) -- ++ (2,0.15) -- ++(0,2);
\node[anchor = north] at (0,-4.15) {$v$};
\node[anchor = north] at (2,-4.15) {$w$};

\draw[semithick] (4,0) -- ++(0, -2) -- ++ (2, 0) -- ++(0,2);
\draw[semithick,postaction={decorate}] (4,-4.15) -- ++(0, 2) -- ++ (2,0) -- ++(0,-2);
\node[anchor = north] at (4,-4.15) {$v$};
\node[anchor = north] at (6,-4.15) {$w$};
\end{tikzpicture}
\caption{An illustration of a \emph{cross} on the left and a \emph{double bar} on the right, in each case connecting two vertices $v$ and $w$. The arrows indicate the path followed till time $T$ by a particle starting at $v$ and initially moving vertically. Note that though we have drawn arrows on the cross/double bar, in the model the particle crosses either type of bridge instantaneously.}\label{fig.bridge-types}
\end{figure}
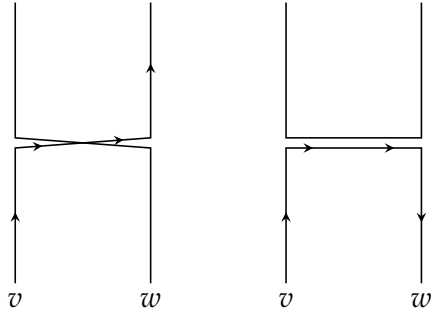

Supposing, as we will, that $G$ is rooted, we will call the meander $X_{u,T}$ {\em recurrent} if, when $X$ is begun with vertex component equal to the root, the process has probability one to visit its initial location at some positive time; in the other case, the meander will be called {\em transient}. For a regular rooted tree (indeed, for any connected graph of bounded degree), these two conditions are easily seen to be characterised by the almost sure presence, or respectively absence, of an infinite cycle in the associated random permutation $\sigma_T$. 

  Bj\"ornberg and Ueltschi proved that, on a regular rooted tree, each of whose vertices has $d$ offspring, there is a value $T_c = T_c(u,d) \in (0,\infty)$ which verifies
\begin{equation}\label{e.tcform}
T_c(u,d) = \frac1d + \frac{1-u(1-u)-\frac16(1-u)^2}{d^2} + o(d^{-2})
\end{equation}
  such that cyclic-time random meander $X_{u,T}$ is transient
when $T \in ( T_c , T_c + A d^{-2} )$  and recurrent when $T < T_c$.
Here, the parameter $A > 0$  is given, and the result is valid when $d$ exceeds a certain value that may depend on $A$. What Bj\"ornberg and Ueltschi demonstrate, then, is the presence of a critical value for the transition from recurrent to transient behavior, at least locally near the value. It remains possible in principle that recurrent behavior may be reestablished as $T$ increases over the putative critical value by an amount whose order exceeds $d^{-2}$.

As Bj\"ornberg and Ueltschi have noted, the coefficient of $d^{-2}$ in~(\ref{e.tcform}), viewed as a function of $u$, has interesting qualitative similarities with behavior witnessed in numerical studies of $\cytrm{u, T}$ on $\Z^d$. The shared features are illustrated in Figure~\ref{f.plot}: convexity, a minimum in $(0,1)$, and a higher value at $u=1$ than at $0$. See \cite{numeric_study} for details.

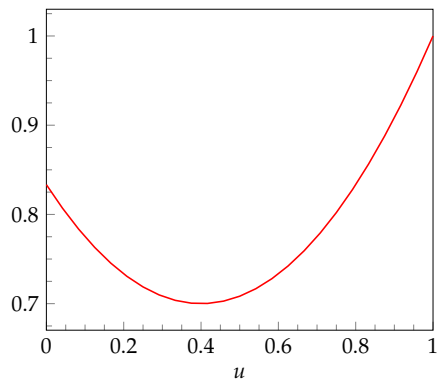
\begin{figure}[h]
\centering
\begin{tikzpicture}[scale=0.75]
    \begin{axis}[domain=0:1, xmin=0, xmax=1, minor tick num = 3, tick pos = left, xlabel=$u$]
    \addplot[no marks, thick, red]{1-x*(1-x)-(1-x)^2/6}; 
    \end{axis}
\end{tikzpicture}
\caption{A plot of $[0,1] \to [0,\infty): u \mapsto 1-u(1-u)-\frac16(1-u)^2$. The plot bears qualitative similarities to one obtained numerically in \cite{numeric_study} for the presumably critical $T$-value for $\cytrm{u, T}$ on $\Z^d$, namely convexity, a unique minimum in $(0,1)$, and a higher value at 1 than 0.}\label{f.plot}
\end{figure}


\subsection{Main result}
As Bj\"ornberg and Ueltschi did, we will consider the graph $G$ to be a tree where each vertex has $d$ offspring.
(One result will be valid when each vertex has at least $d$ offspring.)
Our main result demonstrates that $T_c(u,d)$
is indeed the critical point for the transition from recurrence to transience.

%

\begin{theorem} \label{main_thm}
\begin{enumerate}
\item
There exists $d_0 \in \N$ such that, if $G$
is a rooted tree of bounded degree each of whose offspring has at least $d_0$ offspring, and $u \in [0,1]$, then there exists $T_0 \in (0,\infty)$
such that
CyTRM($u,T$) is
transient when $T > T_0$.
In particular, we may take $d_0=16$ and $T_0=0.495$.
\item If $G$ is instead chosen to be a rooted tree each of whose offspring has exactly $d$ offspring, with $d\geq 56$ and $u\in[0,1]$, 
then there exists a $T_c = T_c(u,d)$ such that CyTRM($u,T$) is transient for $T>T_c$ and recurrent for $T<T_c$. 
The critical value $T_c = T_c(u,d)$
satisfies~(\ref{e.tcform}); it exceeds $d^{-1} + \frac12 d^{-2}$ for $d\geq 56$.
\end{enumerate}
%
\end{theorem}


\subsection{Method of proof: a patchwork of four pieces}

Theorem~\ref{main_thm} is a consequence of four techniques of proof that have been employed to investigate the problem. In order to offer an overall orientation to the reader, we summarise these four methods now, all of which have been employed thus far only in the random stirring case when $u=1$. We list them roughly in increasing order for the ranges of $T$ which the methods address.
The graph $G$ in question is the regular tree with offspring degree $d$.

\noindent{\bf I: Absence of large cycles via percolation.}
The first argument is very simple. If the pole height satisfies $T \in (0, \log \tfrac{d}{d-1})$, then the probability that a given edge supports either a cross or a double bar is less than $d^{-1}$, the bond percolation critical value for $G$. The meander remains among edges of a single such percolation component and is therefore recurrent.

\noindent{\bf II: Angel's argument, slightly above the critical value.} Angel specifies a local configuration 
which forces the meander away from the root. More precisely, he defines a local configuration such that if a particle encounters it, it will either never return to its current position (so its path is transient), or will move to an \emph{offspring} vertex where the local configuration has a chance of being repeated. Angel proves that the vertices enjoying the local configuration form a super-critical Galton-Watson tree, which, along with the above claim, gives transience.
The local configuration depends upon there being a low number of bridges on the pole in question; this event only has a reasonable probability for $T$ slightly above the critical point. As such, the argument can give transience only for $T\in [d^{-1}+(7/6+\varepsilon)d^{-2}, 1/2]$ for $d$ high enough (depending on $\varepsilon$).


\noindent{\bf III: Monotonicity around the critical value.} In \cite{hammond2}, it is argued that, for $T \in (d^{-1}, d^{-1}+2d^{-2}]$, meander transience at $T$ implies transience at any higher value on this interval. In brief, this is accomplished by proving a formula similar to Russo's formula from percolation theory \cite[Theorem 2.25]{grimmett-percolation} regarding the effect on the particle's trajectory of the placement,  uniformly at random, of a single extra bridge on the poles up to distance $n$ from the root.

\noindent{\bf IV: Large cycles, high above the critical value.} In \cite{hammond1}, an argument was presented for transience which works well at high values of $T$: for example, when $d\geq 39$ and $T \geq 429d^{-1}$. This relies on finding a favourable collection of bridges, called ``useful bridges'', whose probability of occurrence does not decay with $T$. The useful bridges serve two purposes: they are locations from where the particle enters unencountered territory (and hence independence comes to the aid of the analysis), and they are also obstacles which the particle must recross back to the root if it is to not be transient. The bulk of the proof is in establishing a linear rate at which useful bridges are generated, thereby showing that infinitely many are generated over the course of the trajectory with positive probability.

Bj\"ornberg and Ueltschi use a different argument in \cite{ueltschi}. We make no use of this argument, except in order to assert the asymptotic formula~(\ref{e.tcform}); the statement of their result is included below for completeness as Proposition \ref{p.near critical point}.

One of the roles of this article is the adaptation and simplification of argument IV as given in~\cite{hammond1} to the $u\neq 1$ setting. The simplification is achieved by the identification of an event which guarantees the generation of a given number of useful bridges; we will describe the event in \Cref{main_proof}. Another improvement is that we have obtained tighter bounds on $d$ for which the argument as a whole is applicable. We elaborate on both these points in greater detail at the end of Section \ref{proof_outline}.


The proof of Theorem~\ref{main_thm}
uses the four arguments mentioned here in a patchwork manner, so that every value $T \in (0,\infty)$
is treated by at least one argument. 
Our task is to adapt the techniques to work in the case when $u \in [0,1]$ is not one.
Indeed, the next four-part result indicates the inference that we will respectively make from each adapted argument.

\begin{prop}\label{p.fourpart}
Let $G$ be a rooted infinite regular tree each of whose offspring has exactly $d$ offspring, and $X = X_{u,T}$ be CyTRM on $G$ with parameters $u\in [0,1]$ and $T\in (0,\infty)$. Then:
\begin{enumerate}
\item If $T\in(0,\log\frac{d}{d-1})$, $X$ is recurrent.

\item If $d\geq 56$ and $T\in [d^{-1} + 2d^{-2}, 4d^{-1}]$, or if $d\geq 9$ and $T\in[4d^{-1},1/2]$, then $X$ is transient.

\item Let $d\geq 26$ and let $T,T'$ be such that $d^{-1} <T <T' \leq d^{-1}+2d^{-2}$.  If $X_{u,T}$ is transient then so is~$X_{u,T'}$.

\item If $d\geq 16$ and $T\in[0.495,\infty)$, then $X$ is transient. This remains true if we relax our hypothesis to every vertex of $G$ having at least $d$ offspring.
\end{enumerate}
\end{prop}

\begin{proof}[Proof of Theorem~\ref{main_thm}]
Part (1) follows immediately from Proposition \ref{p.fourpart} (4). For part (2), we simply write out the ranges guaranteed by Proposition \ref{p.fourpart} and check that they overlap.

\begin{itemize}
\item From Proposition \ref{p.fourpart} (1), we get recurrence for $T< \log\frac{d}{d-1}$, which implies the same for $T\leq d^{-1}+\frac12d^{-2}$.

\item From Proposition \ref{p.fourpart} (2), we get transience for $d\geq 56$ and
$$T\in[d^{-1}+2d^{-2}, 4d^{-1}],$$
as well as for $d\geq 9$ and
$$
T\in \big[ 4d^{-1}, \tfrac12 \big] \, .
$$

\item From Proposition \ref{p.fourpart} (4), we get transience for $d\geq 16$ and
$$T\in[0.495, \infty).$$

\item The range excluded by the above three bullet points is $(d^{-1} + \frac12d^{-2}, d^{-1}+2d^{-2})$. Now if $d\geq 26$, Proposition \ref{p.fourpart} (3) gives monotonicity in $(d^{-1}, d^{-1} + 2d^{-2}]$. Thus if we further have $d\geq 56$ for the second bullet point to apply, the existence of a critical $T_c > d^{-1} + \frac12d^{-2}$ is implied.
\end{itemize}

The claim regarding the asymptotic formula of $T_c$ is a just the formula from the result of Ueltschi and Bj\"ornberg in \cite{ueltschi} referenced above. We include it as the next  proposition without proof for the reader's convenience. This completes the proof of Theorem \ref{main_thm}.
\end{proof}

\begin{prop}[Theorem 1.1 of \cite{ueltschi}] \label{p.near critical point}
Let $A>0$ be given. Then there exists a $d_0$, possibly depending on $A$, such that for $d\geq d_0$, there exists a $T_c = T_c(u,d)$ with the property that CyTRM($u,T$) is transient for $T\in (T_c, T_c+Ad^{-2})$ and recurrent for $T<T_c$. Furthermore, $T_c(u,d)$ satisfies \eqref{e.tcform}.
\end{prop}

It remains, of course, to prove Proposition~\ref{p.fourpart}. 
%
Leaving aside the trivial first argument, the work of adapting arguments II and III
is straightforward. We will not rewrite these arguments, but rather indicate the necessary changes to the original papers in the final \Cref{modifications}. \Cref{p.fourpart} (4) entails more substantial adaptation of the argument given in \cite{hammond1}. We choose to present a self-contained proof of this result, and do so next, in \Cref{main_proof}.

We conclude this section by stating and proving a proposition which will be required in proofs of several parts of \Cref{p.fourpart}. It states that the particle cannot move both vertically up and down (at different times) on any portion of a pole, in spite of the direction-switching double bars.

\begin{prop} \label{no direction conflict}
Suppose a particle performing $\cytrm{u,T}$ on a tree is at the position $(v, t)\in V\times [0,T)$ and moving upwards. Then at any future time at which it is again at $(v, t)$ it will be moving upwards, i.e. it cannot be present at the same location moving in the opposite direction.
\end{prop}
\begin{proof}
Suppose to the contrary that the particle starting at $(v, t)$ and moving up returns to $(v, t)$ while moving down after tracing out some path. This means that the path between the two visits to $(v,t)$ is finite. 

Let $(e, t')$ be the first bridge encountered from cyclic motion upwards from $(v,t)$, connecting to $(v', t')$. Observe that for the particle to be at $(v,t)$ and moving downwards later, it must necessarily come by travelling across $(e, t')$ and then travelling downwards. Thus in our path we have a pairing between two trips across the bridge $(e, t')$. Now observe that we have another path from $(v', t'\pm)$ (depending on whether $(e,t')$ was a cross or a double bar) to itself with the initial and final directions opposite. By induction, we thus have that every bridge in the original path was traversed an even number of times.

If particular, double bars were traversed an even number of times, implying that the direction at $(v, t)$ finally must be the same as it was initially, a contradiction.
\end{proof}




\noindent{\bf Acknowledgment.} The first author thanks Daniel Ueltschi for useful discussions.

\section{Transience for high $T$ for $u\in[0,1]$}\label{main_proof}

We now turn to the main technical element of this paper, the proof of \Cref{p.fourpart} (4). Our argument bears strong similarities to that given in \cite{hammond1}, e.g., it uses the same notion of a useful bridge. However, the proof given here, apart from applying even when $u\neq 1$, is also simpler in certain ways. 
The differences between our argument and \cite{hammond1}'s will be discussed at the end of the proof outline below.


\subsection{Outline of Proof}\label{proof_outline}
We are trying to prove that for sufficiently large $T$, $\cytrm{u, T}$ escapes to infinity with positive probability. This is of course true if $T = \infty$, as the process is then just simple continuous-time random walk on a tree. Here the problem is that we do not necessarily have that each move the particle makes is independent of the past; if the particle returns to a portion of the environment it has already visited, its motion will be ``deterministic'' in that it is determined by the past.

However, each time the particle moves into unvisited territory it gets a new lease of independence which we can exploit. Our approach is to show that, with high probability, these ``frontier departures'' (\Cref{frontier time}) into new territory occur often enough, and detrimental returns to explored territory can be controlled. 
For the analysis of the occurrence of frontier departures we need the notion of useful bridges (\Cref{useful bars}) till a given time $t$, defined according to the particle's trajectory up to time $t$. An important property of these bridges is that their supporting edges have been crossed only once up till time $t$, and so by the tree geometry, if the particle is to be recurrent it must recross all edges supporting useful bridges on its journey back to the root---but useful bridges will be defined such that when the particle has the opportunity to make such a recross, it may instead make a frontier departure into new territory, an event whose probability is bounded below  in \Cref{frontier departure}.

If such a favourable frontier departure occurs, \Cref{ANT} identifies an event which leads to a fixed number of useful bridges being encountered immediately after, and gives a lower bound on its probability. 
If the particle's trajectory is not so favourable and a frontier departure is not made, \Cref{bad return} limits the damage done by bounding how many useful bridges can be lost, roughly speaking.
Our proof will conclude by showing that overall the number of useful bridges grows to infinity with positive probability, which implies that the particle escapes to infinity with positive probability. This is done by a comparison of the number of useful bridges with a suitable random walk on $\Z$ which is made to escape to $+\infty$.

In adapting certain arguments from \cite{hammond1}, we make use of \Cref{no direction conflict} in the proof of \Cref{frontier departure}.

Apart from adaptations, our proof also differs from \cite{hammond1} in ways that result in a shorter and simpler argument. 
Notably, our lower bound on the number of useful bridges is obtained by introducing the \emph{move-forward} event $MF_{N,T}$ that the particle moves away from the root $N$ times consecutively within a time span of $T$ (in this interval we are guaranteed independence no matter the motion). This is a simple event which streamlines the analysis. In the proof of \Cref{p.fourpart} (4), we will be finding a lower bound on the probability of the occurrence of $MF_{N,T}$
in the immediate aftermath of a frontier departure. 
If the event occurs, the delay after the frontier departure at which it is confirmed to do so is a stopping time. The counterpart to this stopping time in~\cite{hammond1}
was a deterministic duration; this convenient new use of randomness is a source of simplification.

In terms of results, we obtain transience for $d\geq 16$ and $T\geq 0.495$, while \cite{hammond1} does so for $d\geq 39$ and $T>429d^{-1}$; however, as \Cref{r.N-choice} observes, we may also get a similar range for higher $d$ by picking parameters $T$ and $N$ differently.

\subsection{Proofs}
Throughout this section our graph $G$ is an infinite tree where each vertex has at least $d$ offspring. We start by establishing some notation.

\begin{notation}
Given the parameter $T>0$, $\P_T$ will denote the probability measure with respect to which the rate one Poisson process on $V\times [0,T)$ is defined.

We will refer to a particle being on the pole of a vertex $v\in V$ at a height $t\in[0,T)$ by the coordinates $(v,t)$. We similarly refer to a bridge supported by an edge $e\in E$ at height $t\in[0, T)$ as $(e,t)$.

The CyTRM($u,T$) process started at $(\phi,0)$, where $\phi$ is the root, will be denoted by $X$, so that $X(t)\in V\times[0, T)$ is the position of the particle at time $t\in [0,\infty)$ and $X(0)=(\phi,0)$. $Y$ will denote the projection of $X$ onto the vertex set $V$, so that $Y(t)$ is the vertex whose pole $X(t)$ is at. We will adopt the intuitive notation that for any $t>0$,
$$X_{[0,t)} = \{X(s) \mid s\in[0,t)\},$$
with the obvious analogue for $Y$.

We use notation for two types of hitting times for $A\subseteq V\times[0,T)$:
$$H_A = \inf\{s\geq 0\mid X(s) \in A\} \quad \text{and}\quad H_{t,A} = \inf\{s\geq t\mid X(s) \in A\}.$$
Further, if $A=\{x\}$, we will replace $A$ in the above notation by $x$.

For an edge $e$ connecting vertices, $e^+$ will denote the vertex $e$ is incident to which is closer to the root and will be called the parent vertex of $e$. Similarly $e^-$ will denote the incident vertex further from the root, called the offspring vertex of $e$. We will refer to the parent vertex of any given vertex of the tree to mean the neighbour closer to the root, and likewise for offspring vertices. The graph distance metric will be denoted by $\mathrm{dist}$.
\end{notation}

We record a simple observation regarding the conditional distribution of the unexplored environment given the trajectory up to time $t$. We will make use of this lemma without comment in the sequel.

\begin{lemma}\label{l.conddist}
Let $t > 0$.
 Consider the law $\P_T$ given $X:[0,t] \to V \times [0,T)$. Let $\mathrm{Found}_t \subseteq E \times [0,T)$ 
denote the set of bridges that $X$ has crossed during $[0,t]$, and let $\mathrm{UnExplored}_t\subseteq E\times [0, T)$ be all elements of $E\times[0, T)$ neither of whose endpoints lie in $X_{[0, t]}$. 
Then the distribution of the collection of bridges $\mathcal B$ given $X_{[0,t]}$ is given by $\mathrm{Found}_t \cup \mathcal B_{(t,\infty)}$, where $\mathcal B_{(t,\infty)}$ is a collection of bridges distributed as a Poisson process on $E\times [0,T)$ with intensity $\one_{\mathrm{UnExplored}_t}$ with respect to product Lebesgue measure.

\end{lemma}
\begin{proof}
It is obvious that $\mathrm{Found}_t$ is contained in $\mathcal B$ as it is known given $X_{[0,t]}$. From the independence property of Poisson processes it follows that the distribution of the remaining bars, i.e. those in $\mathrm{UnExplored}_t$, is unaffected.
\end{proof}

\begin{defn}[Useful bridges]\label{useful bars}
We define, for $t>0$, a set  $\U_t\subseteq E\times[0,T)$  of \emph{useful bridges at time~$t$}. A bridge $(e,s)\in E\times[0,T)$ belongs to $\U_t$ if 
\begin{itemize}
\setlength\itemsep{0.8pt}
\item $H_{e^+}<H_{e^-} < t$,
\item $H_{e^-} - H_{e^+} < T/2$,
\item $\{\tilde t\in[0,t] : Y(\tilde t) = e^+\} = [H_{e^+}, H_{e^-})$, and
\item $\{\tilde t\in[0,t] : Y(\tilde t) = e^-\}$ is an interval with right endpoint strictly less than $t$.
\end{itemize}
\end{defn}

Thus, a bridge is useful at time $t$ if it has been crossed before that time, the particle has spent at most time $T/2$ at the bridge's parent vertex, has visited the parent and offspring vertices only once, and is not at the offspring vertex at time $t$.


\begin{defn}[Frontier time]\label{frontier time}
A time $t>0$ is called a \emph{frontier time} if $Y(t)\not\in \{Y(s) \mid 0\leq s< t\}$, i.e., the vertex whose pole $X(t)$ is at has not been visited before time $t$. 
\end{defn}

\begin{defn}[Frontier departure]
Under $\P_T$ given $X:[0,t]\to V\times [0, T)$, if $(e,s)\in\U_t$ and conditional upon $H_{t,e^-} < \infty$, we say $X$ makes a \emph{frontier departure} from $e$ if after time $H_{t,e^-}$, at the moment of departing $\{e^+, e^-\}$, $X$ arrives at the pole of a vertex it has not visited before. 
\end{defn}

Note carefully that in the above definition we are considering the moment of departure from $\{e^+, e^-\}$, and not from $e^-$ alone; so the particle may go from $e^-$ to $e^+$ first and then depart to an unvisited vertex as part of a frontier departure.


\begin{lemma} \label{frontier departure} 
Let $t>0$. Consider the conditional distribution of $\P_T$ 
 given $X_{[0, t]}$. Let $(e,s)\in \U_t$ with $e^+\neq \phi$ be chosen measurably with respect to $X_{[0,t]}$, and condition further on $H_{t, e^-} < \infty$. Then the probability of making a frontier departure is at least
$\frac{d-1}{d+1}(1-e^{-(d-1)T/2})$.
\end{lemma}

\begin{proof}
Since $(e,s)\in \U_t$, we have that $e^-$ has been visited by time $t$. By the conditioning that $H_{t,e^-}<\infty$, we note that on the particle returning to the pole at $e^-$, it can either stay on the pole till reaching $(e,s)$ (in which case it will jump back to $e^+$), or it can jump to another vertex before reaching the bridge $(e, s)$. Let $J$ be the {\em  jump} event that jumping to $e^+$ occurs, and let $J^c$ be the complementary event.

To simplify notation, let $\tau$ be the hitting time of $\{e^+, e^-\}^c$ after time $t$, i.e.,
$$\tau = H_{t,\{e^+,e^-\}^c}.$$

We analyse the case where $J^c$ occurs first. Since $(e,s)\in \U_t$, $H_{t,e^-}$ is the time of first return of the particle to the pole at $e^-$. So by \Cref{no direction conflict}, the particle is travelling in an unexplored portion of the pole. Further, since $J^c$ occurring is equivalent to there being a bridge on $e^-$ different from $(e,s)$, we only need  consider where it connects to: obviously there are $d-1$ choices of unexplored vertices out of $d+1$ neighbours, and thus we have
$$\P_T\left(Y(\tau) \not\in Y_{[0,\tau)} \mid X_{[0,t]}, J^c\right) \geq \frac{d-1}{d+1}.$$
Now suppose $J$ occurs, i.e., the particle travels back to $e^+$ via the bridge $(e,s)$. Again by the definition of $\U_t$ and \Cref{no direction conflict}, the particle travels in the direction of unexplored area on the pole at $e^+$. The unexplored portion of the pole has length at least $T/2$ since $(e,s)\in U_t$ implies the explored interval has length at most $T/2$. So, conditioned on there being at least one bridge in this unexplored portion, we need to consider the probability that it connects to an unexplored vertex. Doing so and multiplying by the probability of the conditioning event, 
$$\P_T\left(Y(\tau) \not\in Y_{[0,\tau)} \mid X_{[0,t]}, J\right) \geq \frac{d-1}{d+1}\left(1-e^{-(d+1)T/2}\right).$$
Combining the above two gives the lemma.
\end{proof}

In the next lemma we define the {\em move-forward} event $MF_{N,T}$ described earlier and get a lower bound on its probability. This event is the main source of simplification of our proof in comparison to \cite{hammond1}. Though in some sense it is quite a crude event, its job is to generate useful bridges, and it turns out that this is sufficient for our purpose.

\begin{lemma} \label{ANT}
Let $X$ be a $\cytrm{u,T}$ started at $(v,t_0)$ and $MF_{N,T}$ be the \emph{move-forward} event that the particle goes forward at least $N$ times consecutively in the time interval $(0,T)$. On $MF_{N,T}$, let $\tau$ be the random time at which the $N$\textsuperscript{th} consecutive bridge is crossed, i.e.,  $\tau = \inf\{s\geq 0 \mid \mathrm{dist}(Y(s), v) = N\}$. Then we have
$$\P_T(MF_{N,T}) \geq \left(1-\frac{1}{d+1}\right)^N\left[1- e^{N-(d+1)T}\left(\frac{(d+1)T}{N}\right)^N\right] =: p^{(1)}_{N,T,d}.$$
Further, on the event $MF_{N,T}$, $|\mathcal U_\tau| \geq N-2$ a.s. Also, if we condition on $t$ being a frontier time and on $X_{[0,t]}$, the above event with the time interval $(0,T)$ replaced by $(t,t+T)$ occurs with the same probability, and on that event $|\U_{\tau_t}| \geq N-2$ a.s., where $\tau_t = \inf\{s\geq t \mid \mathrm{dist}(Y(s), Y(t)) = N\}$.
\end{lemma}

\begin{proof}
Recall that the gap distribution of a Poisson process of parameter $d+1$ is Exp($d+1$). Let $\xi_1,\ldots, \xi_N$ be iid Exp($d+1$) random variables. By iteratively conditioning on moving forward one step and using the independence obtained by moving forward (regardless of vertical direction of motion), we obtain
\begin{equation}\
\P_T(MF_{N,T}) \geq \left(1-\frac{1}{d+1}\right)^N\cdot \P(\xi_1+\ldots + \xi_N \leq T).
\end{equation}
We need an upper bound on $\P_T(\xi_1+\ldots + \xi_N > T)$. Exponentiating, using the Markov inequality, and recalling that the moment generating function of Exp($d+1$) is given by $f(\lambda) = \frac{d+1}{d+1-\lambda}$, we get
%
%
%
$$\P_T(\xi_1+\ldots + \xi_N > T) \leq e^{-\lambda T}\left(\frac{d+1}{d+1-\lambda}\right)^N.$$
This is minimised when $\lambda = d+1-N/T$, which gives
$$\P_T(\xi_1+\ldots + \xi_N > T) \leq e^{N-(d+1)T}\left(\frac{(d+1)T}{N}\right)^N.$$
Substituting back in (1) implies the claimed lower bound. From the definition of $\U_\tau$, it follows that the last bridge is not in $\U_\tau$ as the particle has not yet left the offspring vertex of the last bridge. Of the remaining $N-1$ bridges, we must exclude any where the particle spent more than $T/2$ time at the offspring vertex before jumping; however, this can happen at most once in a time interval of length $T$.

The fact that the same is true in the time interval $(t, t+T)$ when conditioned on $t$ being a frontier time is straightforward, since the particle is in unexplored territory. More precisely, conditional on $X_{[0,t]}$ and $t$ being a frontier time, \Cref{l.conddist} implies that the distribution of the bridge locations on the pole of $Y(t)$ remains unchanged, and so the above argument applies directly.
\end{proof}

The next lemma considers the situation where the particle returns and does not make a frontier departure, so that Lemmas \ref{frontier departure} and \ref{ANT} do not apply. Its role is to control the damage in this situation by bounding the number of useful bridges, which may be viewed as obstacles to the particle returning to the root, that can be undone.

\begin{lemma} \label{bad return}
Let $t>0$, and $e\in \mathcal U_t$ be the bridge last crossed in $X_{[0,t]}$. Let $p(e^+)$ be the parent of $e^+$. Then, conditionally on $e^+\neq \phi$ and $H_{t,p(e^+)} < \infty$, we have that $|\mathcal U_t\setminus \mathcal U_{H_{t, p(e^+)}}| \leq 2$ a.s. 	 	
\end{lemma}

\begin{proof}
We write $\underline \U_t$ for the set of edges that support a bridge in $\U_t$. Note that for each $t > 0$, these two sets are in one-to-one correspondence by the definition of $\U_t$: no two elements in $\U_t$ can be supported on the same edge. Hence, it suffices to derive the statement of the lemma with $\U_t$ replaced by $\underline \U_t$.

The fourth requirement in the definition of $\smash{\U_{H_{t,p(e^+)}}}$ and $\U_t$ gives that the only way an edge $f\in\underline \U_t$ will not be in $\smash{\U_{H_{t,p(e^+)}}}$ is if the particle visits $f^-$ in $(t, \smash{H_{t,p(e^+)}})$. Note that $\smash{H_{t,p(e^+)}}$ is not included. Now by the tree's geometry, the only edges of $\underline \U_t$ whose children may be visited in $(t, H_{t,p(e^+)})$ are $e$ and $(p(e^+), e^+)$. Thus the lemma follows.
\end{proof}

\begin{defn}[Acceptable return]\label{acceptable return} 
Let $t>0$ and let the bridge $(e,s)$ belong to $\U_t$. Write $e^c := V(G)\setminus\{e^+, e^-\}$. If $X$ returns to $e$ after time $t$, i.e., $H_{t, e^-}<\infty$, we say the return is acceptable if 
\begin{enumerate}
\setlength\itemsep{0.8pt}
\item $X$ makes a frontier departure from $e$ (i.e., $X$ leaves to a previously unencountered vertex), say at time $\tau$ (so that $\tau$ is a frontier time);
\item and $X$ goes $N$ steps forward consecutively in the time interval $(\tau, \tau + T)$ (note that by right-continuity of $X$ the frontier departure step cannot be counted towards $N$). 
\end{enumerate}
\end{defn}

\begin{remark}\label{r.cntd2}
Observe that, conditional on $X$ making a frontier departure, \Cref{ANT} says that item 2 above occurs with probability at least $p_{N,T,d}^{(1)}$. Thus we can combine Lemmas~\ref{frontier departure} and~\ref{ANT} to find that, conditional on $H_{t, e^-}<\infty$ for some bridge $(e,s)\in\U_t$,
$$\P_T(\text{return to } e \text{ is acceptable} \mid H_{t,e^-}<\infty) \geq \frac{d-1}{d+1}(1-e^{-(d-1)T/2}) \times p_{N,T,d}^{(1)} =: p_{N,T,d}^{(2)}.$$
\end{remark}

Now we may prove \Cref{p.fourpart} (4). The idea of the proof is to estimate the number of useful bridges at a sequence of random stopping times which we will construct. If we can show that with positive probability this number goes to infinity without ever hitting $0$, then the CyTRM will be seen to be transient---the particle cannot have returned to the root. 

We do this by using \Cref{ANT}, i.e., by considering the event of moving forward consecutively $N$ steps after a frontier departure; each time this occurs, the number of useful bridges at the time of completing the $N$\textsuperscript{th} step increases by at least $N-2$. Similarly we control the effect of bad returns by \Cref{bad return}. Using the bounds on the probabilities of these events, we can stochastically dominate $|\U_t|$ by a random walk on $\Z$ with related transition probabilities. The final step is to analyse for what values of $T,d,$ and $N$ this random walk can be guaranteed to have positive drift, for such a random walk will stay positive forever with positive probability. We now turn to the technical details.

\begin{proof}[Proof of \Cref{p.fourpart} (4)]
We construct a sequence of stopping times $\tau_k$ where we estimate the number of useful bridges. The stopping times will be defined iteratively based on when the particle next returns to the child $e^-$ of a useful bridge $e$ (if it does) and where it jumps to from there. If the return is acceptable, i.e., on returning it moves forward into new territory and then goes $N$ steps forward consecutively, we will have the next stopping time be the moment that it completes the $N$\textsuperscript{th} step. If the return is not acceptable, we will have the next stopping time be when the particle reaches $p(e^+)$ (the worst case), if it does.

Throughout we will need the number of useful bridges to be at least two; this is only a technical requirement to ensure that we have at least one useful bridge not joined to $\phi$, as Lemmas \ref{frontier departure} and~\ref{bad return} assume the parent vertex is not the root. Hence if at any point $|\U_t|<2$, we choose to give up.

Define $\tau_1 = \inf\{t\geq 0: |\U_t|\geq 2\}$. Observe that $\tau_1 < \infty$ with positive probability; otherwise, set $\tau_j = -\infty$ for every $j>1$, as a technical convention to say that we have failed and are giving up. Similarly for $k\in \N^+$, if $|\U_{\tau_k}| \leq 1$, set $\tau_j = -\infty$ for all $j> k$.

Otherwise, denoting the last bridge crossed in $\U_{\tau_k}$ before time $\tau_k$ by $e_k$, define $\chi_k = H_{\tau_k, e_k^-}$. This time is when the particle next returns to $e_k^-$. Now there are a few cases ($\mathrm{dist}$ is the graph distance on $G$):
\begin{itemize}
\item If $\chi_k = \infty$, set $\tau_{j} = \infty$ for $j > k$, as a technical convention to indicate success; the particle's trajectory is transient.
\item If $\chi_k < \infty$ and the return of $X$ to $e_k$ is acceptable, set 
$$\tau_{k+1} = \inf\{t\geq H_{\chi_k, e_k^c} \mid \mathrm{dist}(Y(t), e_k^-) = N\},$$
i.e., the first time after $H_{\chi_k, e_k^c}$ (the frontier departure time) that $X$ reaches $N$ steps forward.
\item If $\chi_k < \infty$ and the return of $X$ to $e_k$ is not acceptable, set 
$$\tau_{k+1} = H_{\chi_k, p(e^+)},$$
which may be infinite, in which case the particle's trajectory is transient.
\end{itemize}

For $k\in\N$, define $u_k = |\U_{\tau_k}|$. Specify three random variables $p_k, q_k, r_k$, defined under $\P_T$ given $X_{[0,\tau_k]}$, as follows. Let $A_k$ be the event that the return to $e_k$ is acceptable, and set
%
\begin{align*}
p_k &= \P_T\left(\chi_k < \infty \mid X_{[0,\tau_k]}\right)\\
q_k &= \P_T\left(A_k \mid X_{[0,\tau_k]}, \chi_k<\infty\right)\\
r_k &=  \P_T\left(H_{\chi_k, p(e^+)}<\infty\mid X_{[0,\tau_k]}, \chi_k<\infty, A_k^c\right).
\end{align*}

Note that $u_k$ is measurable with respect to $X_{[0,\tau_k]}$. Note also that, by the definition of an acceptable return and Lemmas~\ref{ANT} and~\ref{bad return}, the conditional distribution of $u_{k+1} - u_k$ given $X_{[0, \tau_k]}$ and $\{u_k>1\}$ stochastically dominates the law
$$(1-p_k)\cdot\delta_\infty + p_kq_k\cdot\delta_{N-2} + p_k(1-q_k)(1-r_k)\cdot\delta_\infty + p_k(1-q_k)r_k\cdot\delta_{-2},$$
which is parametrized by $(p_k, q_k, r_k)$. Since it is not in our favour if $p_k=1$ or $r_k=1$, this law stochastically dominates the one where $p_k, r_k = 1$.

\Cref{r.cntd2} says that $\P_T(\cdot\mid X_{[0,\tau_k]}, u_k>1)$-a.s.,
$$q_k \geq p^{(2)}_{N,T,d}.$$

In summary, conditional on $X_{[0,\tau_k]}$ and $\{u_k>1\}$, the law of $u_{k+1}-u_k$ $\P_T$-a.s. stochastically dominates the law
$$p^{(2)}_{N,T,d}\cdot\delta_{N-2} + (1-p^{(2)}_{N,T,d})\cdot\delta_{-2}.$$

Now let $Q:\N^+\to\R$ denote the random walk with independent increments whose law is $p^{(2)}_{N,T,d}\cdot\delta_{N-2} + (1-p^{(2)}_{N,T,d})\cdot\delta_{-2}$ and initial condition $Q(1) = 2$. Let $\rho$ be the first time $Q$ goes strictly below 2, and define $Q_*:\N^+\to \R$ by
$$Q_*(i) = \begin{cases}
Q(i) & \text{if } i\leq \rho\\
0 & \text{if } i>\rho
\end{cases}$$
for each $i\in\N^+$.

We now have that conditionally on $\tau_1<\infty$, $\{u_i: i\in \N^+\}$ stochastically dominates $\{Q^*(i) : i\in\N^+\}$. Thus we need to find $N, T,d$ such that with positive probability $Q(i)$ tends to infinity while staying strictly above 1 always. This is satisfied if the drift is positive, which is to say that it suffices to have
\begin{align}
0&<(N-2)p^{(2)}_{N,T,d} -2(1-p^{(2)}_{N,T,d})\nonumber\\
&= Np^{(2)}_{N,T,d} - 2\nonumber\\
&= N\frac{d-1}{d+1}(1-e^{-(d-1)T/2})\left(1-\frac1d\right)^N\left[1- e^{N-(d+1)T}\left(\frac{(d+1)T}{N}\right)^N\right]-2. \label{e.drift}
\end{align}
Note that~(\ref{e.drift}) is increasing in $T$ for $T>N(d+1)^{-1}$; hence if some choice of $(N,T,d)$ works in this range, so will higher values of $T$. Now taking $N=4, T=0.495$  and $d\geq 16$ gives a positive drift by direct calculation. Therefore $u_i$ remains above 1 and goes to infinity with positive probability for $N=4, T\geq 0.495$ and $d\geq 16$, which can occur only if the particle's trajectory is transient, thus proving the theorem.
\end{proof}

\begin{remark}\label{r.N-choice}
Note that it was not necessary to take $N=4$ in the last calculation; this was only done so as to get transience for $T\geq 0.495$, as the interval $(T_c, 0.5]$ is covered by the other parts of \Cref{p.fourpart}.  In fact, we can fix $\varepsilon>0$, take $T=(1+\varepsilon)N(d+1)^{-1}$, and then adjust $N$ and $d$ in order to make \eqref{e.drift} positive. It is easy to check that this is possible by taking $N$ and $d$ sufficiently large. In other words, given an $\varepsilon>0$, there exist high $N$ and $d$ such that $\cytrm{u,T}$ is transient for $T > (1+\varepsilon)N(d+1)^{-1}$.
\end{remark}

We end this section by indicating which of these lemmas and propositions have been taken or modified from \cite{hammond1}. Our \Cref{l.conddist} is Lemma 2.1 of \cite{hammond1}, our \Cref{frontier departure} is Lemma 2.6 of \cite{hammond1}, and our \Cref{bad return} is Lemma 2.14 of \cite{hammond1}.

\section{Modifications of Previous Proofs}\label{modifications}

In this section we show how to modify existing proofs, namely from \cite{angel} and \cite{hammond2}, to complete the proofs of part (2) and (3) of \Cref{p.fourpart}.
Angel's proof in \cite{angel} is applicable to $u=1$ only, but a minor modification using \Cref{no direction conflict} above allows us to extend it.
The argument of Hammond in \cite{hammond2} is applicable even for $u\neq 1$, as elaborated in the appendix of \cite{ueltschi},
but the bounds can be significantly tightened by a more careful calculation. This is needed for the bounds we claim.

We do not provide self-contained proofs because the involved changes are too mundane to warrant doing so. However, we have provided an overview with indications on how to modify the original proofs to apply when $u\neq 1$. The reader may wish to refer to the original papers to get a complete understanding of the argument.

\subsection{Modification of Angel's Proof}

\begin{theorem} \label{angel theorem}
Let $G$ be a regular tree with $d$ offspring at every vertex (so $(d+1)$-regular). Then for any $\varepsilon>0$, there exists $d_0$ such that for $d\geq d_0$ and any $u\in[0,1]$, $\cytrm{u,T}$ on $G$ will be transient for $T\in [d^{-1}+(7/6+\varepsilon)d^{-2}, 1/2].$

In particular, we have that $\cytrm{u,T}$ is transient for $T\in[d^{-1} + 2d^{-2}, 4d^{-1}]$ for $d\geq 56$ and for $T\in[4d^{-1}, 1/2]$ for $d\geq 9$.
\end{theorem}

This is Theorem 3 of \cite{angel}, but for $u\neq 1$. The proof is essentially \cite{angel}'s except for a minor modification. In particular, the proof of a claim in the initial section of the proof of \cite[Theorem 3]{angel} must be modified for $u\neq 1$, which we isolate below as \Cref{l.angel}. 

First we recall some required definitions from \cite{angel}:

\begin{defn}[Good vertex]
Let $v$ be a vertex and $u$ be its parent. We say $v$ is \emph{good} if there is only a single bridge between $u$ and $v$.
\end{defn}

\begin{defn}[Uncovered vertex]
Let $v$ be a good vertex, $u$ be its parent and a good vertex, and $v'$ be a sibling of $v$. 
Call $v$ \emph{covered by $v'$} if the bridges from $u$ to $v'$ cyclically separate (on the pole of $u$)
the unique bridge from $u$ to $v$ and the unique bridge from $u$ to its parent.

We say a vertex $v$ as above is \emph{uncovered} if it is not covered by any of its siblings.
\end{defn}

\begin{lemma}\label{l.angel}
Suppose $v$ is uncovered and the particle reaches $u$, the parent of $v$. 
Then either the particle reaches $v$, or it leaves $u$ at some point and never returns.
\end{lemma}

\begin{proof}
Suppose that the first case does not occur, i.e., the particle does not reach $v$. We must show that $(u, t_u)$ is not part of a finite cycle. Assume without loss of generality that the particle is moving up from $(u, t_u)$. 

Since $v$ is uncovered, we have that if a bridge supported on an edge is present in $I := [t_u, t_v]$, there is no bridge supported by the same edge outside $I$.
The fact that we never reach $v$ implies that the particle can in future be present at vertex $u$ only in the interval $[t_u, t_v]$.
Due to the tree geometry, this tells us that if $(u, t_u)$ is part of a finite cycle, the bridge at $t_u$ must be encountered via downward motion at some point.
But since we initially moved up from $(u, t_u)$, \Cref{no direction conflict} tells us that this is not possible.
\end{proof}

\begin{proof}[Proof sketch of \Cref{angel theorem}]
Given Lemma~\ref{l.angel}, two things are needed to complete the proof. It must be shown that the number of good, uncovered offspring of distinct vertices are independent random variables, so that the good, uncovered connected component containing $\phi$ is a Galton-Watson tree, and then it must be shown that for the claimed choices of $d$ and $T$, this random variable has mean greater than 1---as this is equivalent to the Galton-Watson tree containing with positive probability an infinite path starting at the root.

This is accomplished by the rest of Angel's proof, which goes through even for $u\neq 1$, and so we obtain the same result as there.
Also refer to \cite[Lemma B.2]{hammond1} for the proof of the particular bounds on $d$ and $T$-intervals.
\end{proof}

\begin{proof}[Proof of \Cref{p.fourpart} (2)]
This is the same statement as the second part of \Cref{angel theorem}.
\end{proof}

\subsection{Modification of Hammond's Proof of Monotonicity}
As substantiated in the appendix of~\cite{ueltschi}, Hammond's proof of monotonicity given in \cite{hammond2} works essentially without modification for any $u\in[0,1]$. However, the bound on $d$ obtained in that paper for which the result applies can be easily tightened to get the bounds we claim by a closer examination of the proof and some new calculations. In this subsection we discuss how to go about doing this.

We do not provide a self-contained proof of the monotonicity claimed in \Cref{p.fourpart} (3). Instead, we give a sketch, taking certain claims and inputs from \cite{hammond2} as black boxes. A reader who has not read \cite{hammond2} should be able to read this subsection and obtain a clear idea of the overall proof, modulo certain facts which are stated but not proved here. To obtain a complete proof, it is recommended to the reader to refer to \cite{hammond2} alongside this subsection to get a detailed understanding of where and how the original proof and calculations are modified.

First we recall some notation from \cite{hammond2}. 
The idea of that paper is to consider the probability that $\cytrm{T}$ does not return to the pole over the root $\phi$, denoted $p_\infty(T)$,
and to show that it is non-decreasing in the interval $(d^{-1}, d^{-1} + 2d^{-2}]$ which contains the critical point.
This is done by considering ``local approximations'' $p_n(T)$, defined as the probability that $\cytrm{T}$ ever reaches level $n$ of the tree,
and proving that these functions are differentiable and non-decreasing in the required interval. This is clearly sufficient as $p_n \downarrow p_\infty$ pointwise. Unlike in \cite{hammond2}, here the quantities $p_n$ and $p_\infty$ are defined with respect to $\cytrm{u, T}$ instead of $\cytrm{T}$. 


\begin{proof}[Proof of \Cref{p.fourpart} (3)]
By the above discussion, this is implied by the next proposition.
\end{proof}

\begin{prop}[Modification of Proposition 1.8 of \cite{hammond2}]\label{pn-derivative}
Let $d\geq 26$ and suppose $d^{-1}< T \leq d^{-1}+2d^{-2}$. Then for each $n\geq 1$, $p_n$ is differentiable at $T$ and $\frac{\mathrm dp_n}{\mathrm dT}(T) > 0$.
\end{prop}
\begin{remark}
Proposition 1.8 of \cite{hammond2} is used to prove that paper's Proposition 1.3 (our \Cref{p.fourpart} (3)), and is stated as $\frac{\mathrm dp_n}{\mathrm dT}(T) > \frac{d}{2}e^{-Td}p_n$;
however, as indicated above, the proof of Proposition 1.3 itself only requires  $\frac{\mathrm dp_n}{\mathrm dT}(T) > 0$, 
which is partly what allows us to get a better range for the $d$ where \Cref{p.fourpart} (3) is applicable.
\end{remark}

\begin{notation}
We let $\mathcal T_n$ be the subgraph induced by the vertices within distance $n$ of the root. 
\end{notation}

Using a formula analogous to Russo's formula from percolation theory, we can write the derivative of $p_n$ in terms of certain ``pivotal'' events $P_n^+$ and $P_n^-$. These events are defined in terms of the effect of a ``uniformly added bridge''. To be precise, to the existing random arrangement of bridges, we add one additional bridge $\mathcal A_n$ sampled from normalized Lebesgue measure on $E(\mathcal T_n) \times [0, T)$, independently of the existing bridges, which is direction maintaining with probability $u$ and direction switching with probability $1-u$. 

This new bridge $\mathcal A_n$ can potentially affect the trajectory of the particle. One of three things can happen: the trajectory of the particle originally did not exit $\mathcal T_n$, and now does; the particle did exit $\mathcal T_n$ originally, but no longer does; or finally, the event of the particle exiting $\mathcal T_n$ is unaffected. We denote by $P_n^+$ and $P_n^-$ the events that the first and the second possibilities occur.

We can now state (without proof) an expression for the derivative of $p_n$ in terms of these events.

\begin{lemma}[Lemma 1.7 of \cite{hammond2}] \label{pn-deriv-expression}
For each $n\in \N$, $p_n:(0,\infty) \to [0,1]$ is differentiable; for $T>0$,
$$\frac{\mathrm dp_n}{\mathrm dT}(T) = |E(\mathcal T_n)|\left(\P_T(P_n^+) - \P_T(P_n^{-})\right).$$ 
\end{lemma}

The probability $\P_T(P_n^+) - \P_T(P_n^{-})$ is decomposed as $A_1 + A_2$, where
$$A_1 = \P_T(P_n^+ \cap C \cap B^c) - \P_T(P_n^- \cap C \cap B^c),$$
$$A_2 = \P_T(P_n^+ \cap C \cap B\cap N) - \P_T(P_n^- \cap C \cap B\cap N).$$

Here $C$ is the \emph{crossing} event that the particle reaches $\mathcal A_n$ before exiting $\mathcal T_n$. 
If $C$ occurs, $B$ is the \emph{bottleneck} event that some edge between the root and the parent vertex supporting $\mathcal A_n$ supports a single bridge. Let $b_n$ be such a bridge that is farthest from the root. Suppose both $C$ and $B$ occur, and that the particle's trajectory is periodic. Then it must cross back along $b_n$ after crossing it the first time. The \emph{no escape} event $N$ is the event that the particle, considered from the time it recrosses $b_n$, reaches the root before exiting $\mathcal T_n$. 

These details are provided to give the reader some idea of the original proof and to be consistent with the notation in \cite{hammond2}; we will not actually be needing these details for our modifications.

At this point our next step is to give a lower bound on $A_1.$ This is the content of the next two lemmas. The first is stated without proof, but the second is one where we will need to make a more careful calculation than in \cite{hammond2}.

\begin{notation}
For convenience, we write $\tau = Td$, so that we are interested in $\tau \in[1, 1+2/d]$.
\end{notation}

\begin{lemma}[Lemma 4.3 of \cite{hammond2}]\label{l.A1first}
Suppose that $n\geq 1$, $d\geq 2$, and $T>0$. Then
$$\P_T(P_n^+\cap C\cap B^c)\geq de^{-\tau}\frac{p_{n-1}}{|E(\mathcal T_n)|}.$$
\end{lemma}

\begin{lemma}[Modification of Lemma 4.5 of \cite{hammond2}]\label{l.A1second}
Suppose that $n\geq 1$, $d\geq 2$, and $1\leq \tau\leq 1+2/d$. Then
$$\P_T(P_n^{-} \cap C\cap B^c) \leq \frac{p_{n-1}}{|E(\mathcal T_n)|}\left(\tau(\tau+1) + 6e^{-1}\frac{a^2(4a^2-11a-9)}{(1-a)^3}\right),$$
where $a =\frac{\tau^2e^{2+\frac{\tau}{d}}}{d(e-1)}$.
\end{lemma}

\begin{proof}[Proof sketch]
Here we indicate how to modify the proof of Lemma 4.5 in \cite{hammond2} to obtain our claim. This comes down to explicitly evaluating a certain sum instead of bounding it.
In the proof of \cite[Lemma 4.5]{hammond2}, the following inequality is obtained:
$$\P_T(P_n^-\cap C\cap B^c)\leq \sum_{k=0}^\infty A_{n,k} \, ,$$
where $A_{n,k}$ is a technical quantity which we will not define; for the purpose of getting a tighter bound,  note that it is proven in~\cite{hammond2} that $A_{n, 0} \leq \frac{p_{n-1}}{|E(\mathcal T_n|}\tau(\tau+1)$ and
$$A_{n,k} \leq \frac{6p_{n-1}}{|E(\mathcal T_n)|}e^{-1}\left((1-e^{-\tau})^{-1}e^{\tau/d+1}\tau^2d^{-1}\right)^k (k+2)^2.$$
Now we only need to estimate the sum of the right hand side as $k$ varies from 0 to $\infty$. This is done in \cite{hammond2} by bounding the exponential term using $\tau\leq 2$ and that $(k+2)^2\leq 2^{k+1}$, but a much tighter bound is easily obtainable.

Using that $\tau\geq 1$, we can bound the exponent as
$$\frac{\tau^2e^{1+\tau/d}}{d(1-e^{-\tau})} \leq \frac{\tau^2e^{2+\frac{\tau}{d}}}{d(e-1)} =: a.$$
The identity  $\sum_{k=1}^\infty (k+2)^2 a^k = a(4a^2-11a+9)/(1-a)^3$  is valid for any $-1< a <1$. With this and the bound on the exponent, summing the series gives the claimed upper bound.
\end{proof}

In \cite{hammond2}, a lower bound on $A_2$ is obtained from the bound on $A_1$; that argument goes through even when $u\neq 1$. Hence, we obtain the following:

\begin{lemma}[Proposition 3.2 of \cite{hammond2}]\label{l.A2}
Let $n\geq 2$, $d\geq 26$, $\tau\in[1, 1+2d^{-1}]$. Then
$$A_2 = \P_T(P_n^+ \cap C \cap B\cap N) - \P_T(P_n^- \cap C \cap B\cap N) \geq 0 \, .$$
\end{lemma}

\begin{proof}[Proof of \Cref{pn-derivative}]
\Cref{pn-deriv-expression} asserts that 
$$\frac{\mathrm dp_n}{\mathrm dT}(T) = |E(\mathcal T_n)|\left(\P_T(P_n^+) - \P_T(P_n^{-})\right) = |E(\mathcal T_n)|(A_1+A_2).$$
Using Lemmas \ref{l.A1first}, \ref{l.A1second}, and \ref{l.A2}, we find that
\begin{align*}
\frac{\mathrm dp_n}{\mathrm dT}(T) 
&\geq p_{n-1}\left[de^{-\tau} - \tau(\tau+1) - 6e^{-1}\frac{a^2(4a^2-11a+9)}{(1-a)^3}\right].
\end{align*}
where $a = \frac{\tau^2e^{2+\frac{\tau}{d}}}{d(e-1)}$. We see that the expression is increasing in $d$ for fixed $\tau$ and decreasing in $\tau$ for fixed $d$. Thus it is bounded below by the value at $\tau = 1+2/d$, and numerically it can be verified that this expression (with $\tau = 1+2/d$) becomes strictly positive at $d=26$. Hence it is positive for all $d\geq 26$ and $1\leq \tau \leq 1+2/d$.
\end{proof}

%
%
%
%
%

\bibliographystyle{alpha}
\bibliography{infinite-ray-switching-bars}

\begin{thebibliography}{BBBU15}

\bibitem[AN94]{aizenman}
Michael Aizenman and Bruno Nachtergaele.
\newblock Geometric aspects of quantum spin states.
\newblock {\em Comm. Math. Phys.}, 164(1):17--63, 1994.

\bibitem[Ang03]{angel}
Omer Angel.
\newblock Random infinite permutations and the cyclic time random walk.
\newblock In {\em Discrete random walks ({P}aris, 2003)}, Discrete Math. Theor.
  Comput. Sci. Proc., AC, pages 9--16. Assoc. Discrete Math. Theor. Comput.
  Sci., Nancy, 2003.

\bibitem[BBBU15]{numeric_study}
Alessandro Barp, Edoardo~Gabriele Barp, Fran\c{c}ois-Xavier Briol, and Daniel
  Ueltschi.
\newblock A numerical study of the 3{D} random interchange and random loop
  models.
\newblock {\em J. Phys. A}, 48(34):345002, 12, 2015.

\bibitem[BU16]{ueltschi}
J.~E. {Bj{\"o}rnberg} and D.~{Ueltschi}.
\newblock {Critical parameter of random loop model on trees}.
\newblock {\em ArXiv e-prints}, August 2016.

\bibitem[Fey53]{feynman}
Richard~P. Feynman.
\newblock Atomic theory of the $\lambda$ transition in {H}elium.
\newblock {\em Physical Review}, 91(6):1291, 1953.

\bibitem[Gri99]{grimmett-percolation}
Geoffrey Grimmett.
\newblock {\em Percolation}, volume 321 of {\em Grundlehren der Mathematischen
  Wissenschaften [Fundamental Principles of Mathematical Sciences]}.
\newblock Springer-Verlag, Berlin, second edition, 1999.

\bibitem[Ham13]{hammond1}
Alan Hammond.
\newblock Infinite cycles in the random stirring model on trees.
\newblock {\em Bull. Inst. Math. Acad. Sin. (N.S.)}, 8(1):85--104, 2013.

\bibitem[Ham15]{hammond2}
Alan Hammond.
\newblock Sharp phase transition in the random stirring model on trees.
\newblock {\em Probab. Theory Related Fields}, 161(3-4):429--448, 2015.

\bibitem[Har72]{harris}
T.E. Harris.
\newblock Nearest-neighbor {M}arkov interaction processes on multidimensional
  lattices.
\newblock {\em Advances in mathematics}, 9(1):66--89, 1972.

\bibitem[T\'93]{toth}
B\'alint T\'oth.
\newblock Improved lower bound on the thermodynamic pressure of the spin
  {$1/2$} {H}eisenberg ferromagnet.
\newblock {\em Lett. Math. Phys.}, 28(1):75--84, 1993.

\bibitem[Uel13]{ueltschi-reverse}
Daniel Ueltschi.
\newblock Random loop representations for quantum spin systems.
\newblock {\em J. Math. Phys.}, 54(8):083301, 40, 2013.

\end{thebibliography}
 
 
\end{document}